\documentclass[a4paper,reqno,12pt]{amsart}

\usepackage{geometry,times}
\usepackage{amsmath,amssymb}
\usepackage{hyperref}

\usepackage{color}
\usepackage{graphicx}

\newtheorem{theorem}{Theorem}[section]
\newtheorem{lemma}[theorem]{Lemma}
\newtheorem{corollary}[theorem]{Corollary}

\theoremstyle{definition}

\theoremstyle{remark}

\numberwithin{equation}{section}

\begin{document}

\title{\texttt{LinCode} -- computer classification of linear codes}

\author{Sascha Kurz}
\address{Sascha Kurz, Department of Mathematics, Physics and Informatics, University of Bayreuth, Bayreuth, Germany}
\email{sascha.kurz@uni-bayreuth.de}

\subjclass[2000]{Primary 94B05; Secondary 05E20}

\date{}


\begin{abstract}
  We present an algorithm for the classification of linear codes over finite fields, 
  based on lattice point enumeration. We 
  validate a correct implementation of our algorithm with known classification results from the literature, which we partially extend 
  to larger ranges of parameters.
  
  \medskip
  
  \noindent
  \textit{Keywords: linear code, classification, enumeration, code equivalence, lattice point enumeration}\\
  \textit{ACM:} E.4, G.2, G.4
\end{abstract}

\maketitle

\section{Introduction}
\label{sec_introduction}
Linear codes play a central role in coding theory for several reasons. They permit a compact representation via 
generator matrices as well as efficient coding and decoding algorithms. Also multisets of points in the projective 
space $\operatorname{PG}(k-1,\mathbb{F}_q)$ of cardinality $n$ correspond to linear $[n,k]_q$ codes, see 
e.g.~\cite{dodunekov1998codes}. So, let $q$ be a prime power and $\mathbb{F}_q$ be the field of order $q$. A $q$-ary 
linear code of length $n$, dimension $k$, and minimum (Hamming) distance at least $d$ is called an $[n,k,d]_q$ code. 
If we do not want to specify the minimum distance $d$, then we also speak of an $[n,k]_q$ code or of an 
$\left[n,k,\left\{w_1,\dots,w_l\right\}\right]_q$ if the non-zero codewords have weights in $\left\{w_1,\dots,w_k\right\}$. 
If for the binary case $q=2$ all weights $w_i$ are divisible by $2$, we also speak of an even code.   
We can also look at those codes as $k$-dimensional subspaces of the Hamming space $\mathbb{F}_q^n$. An $[n,k]_q$ code 
can be represented by a generator matrix $G\in\mathbb{F}_q^{k\times n}$ whose row space gives the set of all $q^k$ codewords 
of the code. In the remaining part of the paper we always assume that the length $n$ of a given linear code equals its 
effective length, i.e., for every coordinate there exists a codeword with a non-zero entry in that coordinate. While a 
generator matrix gives a compact representation of a linear code it is far from being unique. Special generator matrices 
are so-called systematic generator matrices, which contain a $k\times k$ unit matrix in the first $k$ columns. If we apply 
row operations of the Gaussian elimination algorithm onto a generator matrix we do not change the code itself but just its 
representation via a generator matrix. Also column permutations or applying field automorphisms do not change the essential 
properties of a linear code. Applying all these transformations, we can easily see that each $[n,k]_q$ code admits an isomorphic 
code with a systematic generator matrix. Already in 1960 Slepian has enumerated binary linear codes for small parameters up to 
isomorphism (or isometry) \cite{slepian1960some}. The general classification problem for $[n,k]_q$ 
codes has not lost its significance since then, see e.g.\ \cite{betten2006error}. In \cite{jaffe2000optimal} all optimal binary 
linear $[n,k,d]_2$ codes up to length $30$ have been completely classified, where in this context optimal means that no $[n-1,k,d]_2$, 
$[n+1,k+1,d]_2$, or $[n+1,k,d+1]_2$ code exists. Classification algorithms for linear codes have been presented in \cite{ostergaard2002classifying}, 
see also \cite[Section 7.3]{kaski2006classification}. A software package \texttt{Q-Extension} is publicly available, see \cite{bouyukliev2007q} 
for a description. The further development to a new version \texttt{QextNewEdition} was recently presented in \cite{bouyukliev2019classification}.
  
The aim of this paper is to present an algorithmic variant for the classification problem for linear codes. It is implemented in 
an evolving software package \texttt{LinCode}. As the implementation of such a software is a delicate issue, we exemplarily verify 
several classification results from the literature and partially extend them. That the algorithm is well suited for parallelization 
is demonstrated e.g.\ by classifying the $1\,656\,768\,624$ even $[21,8,6]_2$ codes. As mentioned in \cite{ostergaard2002classifying}, 
one motivation for the exhaustive enumeration of linear codes with some specific parameters is that afterwards the resulting codes can 
be easily checked for further properties. Exemplarily we do here so for the number of minimal codewords of a linear code, see 
Subsection~\ref{subsec_applications}.

The remaining part of the paper is organized as follows. In Section~\ref{sec_extending} we present the details and the theoretical  
foundation of our algorithm. Numerical enumeration and classification results for linear codes are listed in Section~\ref{sec_results}. 
Finally, we draw a brief conclusion in Section~\ref{sec_conclusion}.

\section{Extending linear codes}
\label{sec_extending}

As mentioned in the introduction, we represent an $[n,k]_q$ code by a systematic generator matrix $G\in\mathbb{F}_q^{k\times n}$, i.e., 
$G$ is of the form $G=\left(I_k|R\right)$, where $I_k$ is the $k\times k$ unit matrix and $R\in\mathbb{F}_q^{k\times (n-k)}$. While this 
representation is quite compact, it nevertheless can cause serious storage requirements if the number of codes get large. Storing all 
generator matrices of the even $[21,8,6]_2$ codes, mentioned in the introduction, needs more than $2.78\cdot 10^{11}$ bits ($1.72\cdot 10^{11}$ 
bits, if the unit matrices are omitted). 

Our general strategy to enumerate linear codes is to start from a (systematic) generator matrix $G$ of a code and to extend $G$ to a 
generator matrix $G'$ of a {\lq\lq}larger{\rq\rq} code. Of course, there are several choices how the shapes of the matrices $G$ and $G'$ 
can be chosen, see e.g.\ \cite{bouyukliev2019classification,ostergaard2002classifying} for some variants. Here we assume the form
$$
  G'=\begin{pmatrix}
    I_k & 0\dots 0  & R \\
    0   & \underset{r}{\underbrace{1\dots 1}} & \star
  \end{pmatrix}
$$ 
where $G=\left(I_k|R\right)$ and $r\ge 1$. Note that if $G$ is a systematic generator matrix of an $[n,k]_q$ code, then $G'$ is a  
systematic generator matrix of an $[n+r,k+1]_q$ code. Typically there will be several choices for the $\star$s and some of these can 
lead to isomorphic codes. So, in any case we will have to face the problem that we are given a set $\mathcal{C}$ of linear codes and 
we have 	to sift out all isomorphic copies. In the literature several variants of definitions of isomorphic codes can be found. Here 
we stick to \cite[Definition 1.4.3]{betten2006error} of linearly isometric codes, i.e., linearity and the Hamming distance between pairs 
of codewords are preserved. This assumption boils down to permutations of the coordinates and applying field automorphisms, see e.g.\ 
\cite[Section 1.4]{betten2006error} for the details. A classical approach for this problem is to reformulate the linear code as 
a graph, see \cite{bouyukliev2007code}, and then to compare canonical forms of graphs using the software package \texttt{Nauty} 
\cite{mckay1990nauty}, see also \cite{ostergaard2002classifying}. In our software we use the implementation from \texttt{Q-Extension} 
as well as another direct algorithmic approach implemented in the software \texttt{CodeCan} \cite{feulner2009automorphism}. In our 
software, we can switch between these two tools to sift out isomorphic copies and we plan to implement further variants. The reason 
to choose two different implementations for the same task is to independently validate results.\footnote{Moreover, there are some technical 
limitations when applying \texttt{Q-Tools} from \texttt{Q-Extension} to either many codes or codes with a huge automorphism group. Also 
the field size is restricted to be at most $4$. As far as we know, the new version \texttt{QextNewEdition} does not have such limitations.}     
 
It remains to solve the extension problem from a given generator matrix $G$ to all possible extension candidates $G'$.  To this end we 
utilize the geometric description of the linear code generated by $G$ as a multiset $\mathcal{M}$ of points in $\operatorname{PG}(k-1,\mathbb{F}_q)$, 
where
$$
  \mathcal{M}=\left\{\left\{  \langle g^i\rangle\,:\, 1\le i\le n  \right\}\right\},\footnote{We use the notation $\{\{\cdot\}\}$ to emphasize that 
  we are dealing with multisets and not ordinary sets. A more precise way to deal with a multiset $\mathcal{M}$ in $\operatorname{PG}(k-1,\mathbb{F}_q)$ 
  is to use a characteristic function $\chi$ which maps each point $P$ of $\operatorname{PG}(k-1,\mathbb{F}_q)$ to an integer, which is the number 
  of occurences of $P$ in $\mathcal{M}$. With this, the cardinality $\#\mathcal{M}$ can be writen as the sum over $m(P)$ for all points 
  $P$ of $\operatorname{PG}(k-1,\mathbb{F}_q)$.} 
$$ 
$g^i$ are the $n$ columns of $G$, and $\langle v\rangle$ denotes the row span of a column vector $v$. In general, the $1$-dimensional subspaces of 
$\mathbb{F}_q^k$ are the points of $\operatorname{PG}(k-1,\mathbb{F}_q)$. The $(k-1)$-dimensional subspaces of $\mathbb{F}_q^k$ are called the 
hyperplanes of $\operatorname{PG}(k-1,\mathbb{F}_q)$. By $m(P)$ we denote the multiplicity of 
a point $P\in\mathcal{M}$. We also say that a column $g^i$ of the generator matrix has multiplicity $m(P)$, where $P=\langle g^i\rangle$ is the 
corresponding point, noting that the counted columns can differ by a scalar factor. Similarly, let $\mathcal{M}'$ denote the multiset of points 
in $\operatorname{PG}((k+1)-1,\mathbb{F}_q)$ that corresponds to the code generated by the generator matrix $G'$. Note that our notion of isomorphic 
linear codes goes in line with the notion of isomorphic multisets of points in projective spaces, see \cite{dodunekov1998codes}. Counting column 
multiplicities indeed partially takes away the inherent symmetry of the generator matrix of a linear code, i.e., the ordering of the columns and 
multiplications of columns with non-zero field elements is not specified explicitly any more. If the column multiplicity of every column is exactly one, 
then the code is called projective.

Our aim is to reformulate the extension problem $G\rightarrow G'$ as an enumeration problem of integral points in a polyhedron. 
Let $W\subseteq \{i\Delta\,:a\le i\le b\}\subseteq \mathbb{N}_{\ge 1}$ be a set of feasible weights for the non-zero codewords, where we assume $1\le a\le b$ 
and $\Delta\ge 1$.\footnote{Choosing $\Delta=1$ such a representation is always possible. Moreover, in many applications we can choose $\Delta>1$ quite 
naturally. I.e., for optimal binary linear $[n,k,d]_2$ codes with even minimum distance $d$, i.e., those with maximum possible $d$,  we can always assume 
that there exists an \emph{even} code, i.e., a code where all weights are divisible by $2$.} Linear codes where all weights of the codewords are divisible 
by $\Delta$ are called $\Delta$-divisible and introduced by Ward, see e.g.~\cite{ward2001divisible,ward1981divisible}. 

The non-zero codewords of the code generated by the generator 
matrix $G$ correspond to the non-trivial linear combinations of the rows of $G$ (over $\mathbb{F}_q$). In the geometric setting, i.e., where an $[n,k]_q$ code $C$  
is represented by a multiset $\mathcal{M}$, each non-zero codeword $c\in C$ corresponds to a hyperplane $H$ of the projective space 
$\operatorname{PG}(k-1,\mathbb{F}_q)$. (More precisely, $\mathbb{F}_q^*\cdot c$ is in bijection to $H$, where $\mathbb{F}_q^*=\mathbb{F}_q\backslash\{0\}$.) 
With this, the Hamming weight of a codeword $c$ is given by $$n-\sum_{P\in\operatorname{PG}(k-1,\mathbb{F}_q)\,:\,P \in \mathcal{M},\,P\le H} m(P),$$ see 
\cite{dodunekov1998codes}. By $\mathcal{P}_{k}$ we denote the set of points of $\operatorname{PG}(k-1,\mathbb{F}_q)$ and by $\mathcal{H}_k$ the set of hyperplanes.

\begin{lemma}
  \label{lemma_ILP}
  Let $G$ be a systematic generator matrix of an $[n,k]_q$ code $C$ whose non-zero weights are contained in $\{i\Delta\,:a\le i\le b\}\subseteq \mathbb{N}_{\ge 1}$. 
  By $c(P)$ we denote the number of columns of $G$ whose row span equals $P$ for all points $P$ of $\operatorname{PG}(k-1,\mathbb{F}_q)$ and set $c(\mathbf{0})=r$ 
  for some integer $r\ge 1$. With this let $\mathcal{S}(G)$ be the set of feasible solutions of 
  \begin{eqnarray}
    \Delta y_H+\sum_{P\in\mathcal{P}_{k+1}\,:\,P\le H} x_P =n-a\Delta&&\forall H\in\mathcal{H}_{k+1}\label{eq_hyperplane}\\
    \sum_{q\in\mathbb{F}_q} x_{\langle (u |q)\rangle } =c(\langle u\rangle ) && \forall \langle u\rangle \in\mathcal{P}_k \cup\{\mathbf{0}\} \label{eq_c_sum}\\
    x_{\langle e_i\rangle}\ge 1&&\forall 1\le i\le k+1\label{eq_systematic}\\
    x_P\in \mathbb N &&\forall P\in\mathcal{P}_{k+1}\\ 
    y_H\in\{0,...,b-a\} && \forall H\in\mathcal{H}_{k+1}\label{hyperplane_var},
  \end{eqnarray}
  where $e_i$ denotes the $i$th unit vector in $\mathbb{F}_q^{k+1}$. Then, for every systematic generator matrix $G'$ of an $[n+r,k+1]_q$ code $C'$ 
  whose first $k$ rows coincide with $G$ and whose weights of its non-zero codewords are contained in $\{i\Delta\,:\, a\le i\le b\}$, we have a solution 
  $(x,y)\in\mathcal{S}(G)$ such that $G'$ has exactly $x_P$ columns whose row span is equal 
  to $P$ for each $P\in\mathcal{P}_{k+1}$.
\end{lemma}
\begin{proof}
  Let such a systematic generator matrix $G'$ be given and $x_P$ denote the number of columns of $G'$ whose row span is  equal to $P$ for all points $P\in\mathcal{P}_{k+1}$. 
  Since $G'$ is systematic, Equation~(\ref{eq_systematic}) is satisfied. As $G'$ arises by appending a row to $G$, also Equation~(\ref{eq_c_sum}) is satisfied 
  for all $P\in \mathcal{P}_k$. For $P=\mathbf{0}$ Equation~(\ref{eq_c_sum}) is just the specification of $r$. Obviously, the $x_P$ are non-negative 
  integers. The conditions (\ref{eq_hyperplane}) and (\ref{hyperplane_var}) correspond to the restriction that the weights are contained in 
  $\{i\Delta\,:\, a\le i\le b\}$.   
\end{proof}
We remark that some of the constraints (\ref{eq_hyperplane}) are automatically satisfied since the subcode $C$ of $C'$ satisfies all constraints on the weights. If there 
are further forbidden weights in $\{i\Delta\,:a\le i\le b\}$ then, one may also use the approach of Lemma~\ref{lemma_ILP}, but has to filter out the integer solutions that 
correspond to codes with forbidden weights. Another application of this first generate, then filter strategy is to remove some of the constraints (\ref{eq_hyperplane}), which 
speeds up, at least some, lattice point enumeration algorithms. In our implementation we use \texttt{Solvediophant} \cite{wassermann2002attacking}, which is based on the 
LLL algorithm \cite{lenstra1982factoring}, to enumerate the integral points of the polyhedron from Lemma~\ref{lemma_ILP}.

Noting that each $[n',k',W]_q$ code, where $W\subseteq \mathbb{N}$ is a set of weights, can indeed be obtained by extending\footnote{This operation is also called \emph{lengthening} 
in the coding theoretic literature, i.e., 
both the effective length $n$ and the dimension $k$ is increased, while one usually assumes that the redundancy $n-k$ remains fix. The reverse operation is called  
\emph{shortening}.} all possible $[n'-r,k'-1,W]_q$ codes via Lemma~\ref{lemma_ILP}, where $1\le r\le n'-k'+1$, 
already gives an algorithm for enumerating and classifying $[n',k',W]_q$ codes. (For $k'=1$ there exists a unique code for each weight $w\in W$, which admits 
a generator matrix consisting of $w$ ones.) However, the number of codes $C$ with generator matrix $G$ that yield the same $[n',k',W]_q$ code $C'$ with generator 
matrix $G'$ can grow exponentially with $k'$. We can limit this growth a bit by studying the effect of the extension operation and its reverse on some code invariants.

\begin{lemma}
  \label{lemma_shortening}
  Let $C'$ be an $[n',k',W]_q$ code with generator matrix $G'$. If $G'$ contains a column $g'$ of multiplicity $r\ge 1$, then there exists a generator matrix $G$ of an $[n'-r,k'-1,W]_q$ 
  code $C$ such that the extension of $G$ via Lemma~\ref{lemma_ILP} yields at least one code that is isomorphic to $C'$.   
  Moreover, if $\Lambda$ is the maximum column multiplicity of $G'$, without counting the columns whose row span equals $\langle g'\rangle$, then the maximum column multiplicity of 
  $G$ is at least $\Lambda$.
\end{lemma}
\begin{proof}
  Consider a transform $\tilde{G}$ of $G'$ such that the column $g'$ of $G'$ is turned into the $j$th unit vector $e_j$ for some integer $1\le j\le k'$. Of course also $\tilde{G}$ is 
  a generator matrix of $C'$. Now let $\hat{G}$ be the $(k'-1)\times (n'-r)$-matrix over $\mathbb{F}_q$ that arises from $\tilde{G}$ after removing the $r$ occurrences of the 
  columns with row span $\langle e_j\rangle$ and additionally removing the $j$th row. Note that the non-zero weights of the linear code generated by $\hat{G}$ are also contained 
  in $W$. If $G$ is a systematic generator matrix of the the linear code $C$ generated by $\hat{G}$, then Lemma~\ref{lemma_ILP} applied to $G$ with the chosen parameter $r$ yields 
  especially a linear code with generator matrix $G'$ as a solution. By construction the effective length of $C$ is indeed $n'-r$. Finally, note that removing a row from a generator 
  matrix does not decrease column multiplicities.    
\end{proof}

\begin{corollary}
  \label{cor_shortening}
  Let $C'$ be an $[n',k',W]_q$ code with generator matrix $G'$ and minimum column multiplicity $r$. Then there exists a generator matrix $G$ of an $[n'-r,k'-1,W]_q$ 
  code $C$ with minimum column multiplicity at least $r$ such that the extension of $G$ via Lemma~\ref{lemma_ILP} yields at least one code that is isomorphic $C'$.
\end{corollary}

Corollary~\ref{cor_shortening} has multiple algorithmic implications. If we want to classify all $[n,k,W]_q$ codes, then we need the complete lists of $[\le n-1,k-1,W]_q$ codes, 
where $[\le n',k',W'_q]$ codes are those with an effective length of at most $n'$. Given an $[n',k-1,W]_q$ code with $n'\le n-1$ we only need to extend those codes which 
have a minimum column multiplicity of at least $n-n'$ via Lemma~\ref{lemma_ILP}. If $n-n'>1$ this usually reduces the list of codes, where an extensions needs to be computed. 
Once the set $\mathcal{S}(G)$ of feasible solutions is given, we can also sift out some solutions before applying the isomorphism sifting step. Corollary~\ref{cor_shortening} 
allows us to ignore all resulting codes which have a minimum column multiplicity strictly smaller than $n-n'$. Note that when we know $x_{P}>0$, which we do know e.g.\ 
for $P=\langle e_i\rangle$, where $1\le i\le k+1$, then we can add the valid inequality $x_P\ge n-n'$ to the inequality system from Lemma~\ref{lemma_ILP}. We call the application 
of the extension step of Lemma~\ref{lemma_ILP} under these extra assumptions \emph{canonical length extension} or \emph{canonical lengthening}.
 
As an example we consider the $[7,2]_2$ code that arises from two codewords of Hamming weight $4$ whose support intersect in cardinality $1$, i.e., their sum has Hamming weight $6$. 
A direct construction gives the generator matrix
$$
 G_1=
  \begin{pmatrix}
    1 & 1 & 1 & 1 & 0 & 0 & 0\\ 
    0 & 0 & 0 & 1 & 1 & 1 & 1 
  \end{pmatrix},
$$
which can be transformed into 
$$
 G_2=
  \begin{pmatrix}
    1 & 1 & 1 & 1 & 0 & 0 & 0\\ 
    1 & 1 & 1 & 0 & 1 & 1 & 1 
  \end{pmatrix}. 
$$ 
Now column permutations are necessary to obtain a systematic generator matrix
$$
 G_3=
  \begin{pmatrix}
    1 & 0 & 0 & 0 & 1 & 1 & 1 \\ 
    0 & 1 & 1 & 1 & 1 & 1 & 1 
  \end{pmatrix}. 
$$ 
Note that $G_2$ and $G_3$ do not generate the same but only isomorphic codes. Using the canonical length extension the systematic 
generator matrix
$$
 G_0=
  \begin{pmatrix}
    1 & 1 & 1 & 1 
  \end{pmatrix} 
$$ 
of a single codeword of Hamming weight $4$ cannot be extended to $G_3$, since we would need to choose $r=3$ to get from a $[4,1]_2$ code to a 
$[7,2]_2$ code, while the latter code has a minimum column multiplicity of $1$. However, the unique codeword with Hamming weight $6$ and 
systematic generator matrix
$$
  G=
  \begin{pmatrix}
    1 & 1 & 1 & 1 & 1 & 1 
  \end{pmatrix} 
$$ 
can be extended to
$$
 G_4=
  \begin{pmatrix}
    1 & 0 & 1 & 1 & 1 & 1 & 1 \\ 
    0 & 1 & 0 & 0 & 1 & 1 & 1 
  \end{pmatrix}, 
$$ 
which generates the same code as $G_3$. So, we needed to consider an extension of a $[6,1]_2$ code to a $[7,2]_2$ code. Now let us dive into the 
details of the integer linear programming formulation of Lemma~\ref{lemma_ILP}. In our example we have $k=1$ and $q=2$, so that 
$\mathcal{P}_1=\left\{\langle (1)\rangle\right\}$, and 
$$
  \mathcal{P}_2=\left\{
  \left\langle\begin{pmatrix}1\\0\end{pmatrix}\right\rangle,\left\langle\begin{pmatrix}0\\1\end{pmatrix}\right\rangle,\left\langle\begin{pmatrix}1\\1\end{pmatrix}\right\rangle  
  \right\}.
$$
The multiplicities corresponding to the columns of $G$ and $r$ are given by
$$
  c(\langle(1)\rangle)=6\quad\text{and}\quad c(\langle(0)\rangle)=1.
$$
Due to constraint~(\ref{eq_c_sum}) we have 
$$
  x_{\langle e_1\rangle}+x_{\langle e_1+e_2\rangle} =6\quad\text{and}\quad x_{\langle e_2\rangle}=1.
$$
Constraint~(\ref{eq_systematic}) reads
$$
  x_{\langle e_1\rangle}\ge 1\quad\text{and}\quad x_{\langle e_2\rangle}\ge 1.
$$
In order to write down constraint~(\ref{eq_hyperplane}), we need to specify the set $W$ of allowed weights. Let us choose $W=\{4,6\}$, i.e., $\Delta=2$, $a=2$, and $b=3$. 
If we label the hyperplanes by $\mathcal{H}=\left\{1,2,3\right\}$, for the ease of notation, we obtain
\begin{eqnarray*}
  2y_1+x_{\langle e_2\rangle} &=& 3,\\
  2y_2+x_{\langle e_1+e_2\rangle} &=& 3,\text{ and}\\
  2y_3+x_{\langle e_1\rangle} &=& 3.
\end{eqnarray*}
Since the $y_i$ are in $\{0,1\}$ we have $x_{\langle e_1\rangle}\le 3$ and $x_{\langle e_1+e_2\rangle}\le 3$, so that 
$x_{\langle e_1\rangle}= 3$ and $x_{\langle e_1+e_2\rangle}=3$. The remaining variables are given by $x_{\langle e_2\rangle}=1$, $y_1=1$, $y_2=0$, and $y_3$. 
Thus, in our example there is only one unique solution, which then corresponds to generator matrix $G_4$ (without specifying the exact ordering of the columns 
of $G_4$). 

Note that for the special situation $k+1=2$, every hyperplane of $\mathcal{P}_2$ consists of a unique point. The set of column or point multiplicities is left invariant 
by every isometry of a linear code. For hyperplanes in $\operatorname{PG}(k+1,\mathbb{F}_q)$ or non-zero codewords of $C'$ a similar statement applies. To this end we
introduce the weight enumerator $w_C(x)=\sum_{i=0}^n A_ix^i$ of a linear code $C$, where $A_i$ counts the number of codewords of Hamming weight exactly $i$ in $C$. Of 
course, the weight enumerator $w_C(x)$ of a linear code $C$ does not depend on the chosen generator matrix $C$. The geometric reformulation uses the number $a_i$ 
of hyperplanes $H\in\mathcal{H}_k$ with $\# H\cap \mathcal{M}:=\sum_{P\in\mathcal{P}_k\,:\,P\in \mathcal{M},\,P\le H} m(P)=i$. The counting vector $\left(a_0,\dots, a_n\right)$ 
is left unchanged by isometries. One application of the weight enumerator in our context arises when we want to sift out isomorphic copies from a list $\mathcal{C}$ of 
linear codes. Clearly, two codes whose weight enumerators do not coincide, cannot be isomorphic. So, we can first split $\mathcal{C}$ according to the occurring 
different weight enumerators and then apply one of the mentioned algorithms for the ismorphism filtering on the smaller parts separately. We can even refine this 
invariant a bit more. For a given $[n,k]_q$ code $C$ with generator matrix $G$ and corresponding multiset $\mathcal{M}$ let $\widetilde{\mathcal{M}}$ be the set of 
different elements in $\mathcal{M}$, i.e., $\#\mathcal{M}=\sum_{P\in\widetilde{\mathcal{M}}} m(P)$, which means that we ignore the multiplicities in 
$\widetilde{\mathcal{M}}$. With this we can refine Lemma~\ref{lemma_shortening}: 

\begin{lemma}
  \label{lemma_shortening_refined}
  Let $C$ be an $[n,k,W]_q$ code with generator matrix $G$ and $\mathcal{M}$, $\widetilde{\mathcal{M}}$ as defined above. For each $P\in \widetilde{\mathcal{M}}$ 
  there exists a generator matrix $G_P$ of an $[n-m(P),k-1]_q$ code such that the extension of $G_P$ via Lemma~\ref{lemma_ILP} yields at least one code that is isomorphic to $C$.  
\end{lemma}

Now we can use the possibly different weight enumerators of the subcodes generated by $G_P$ to distinguish some of the extension paths. 

\begin{corollary}
  \label{cor_shortening_refined}
  Let $C'$ be an $[n',k',W]_q$ code with generator matrix $G'$, minimum column multiplicity $r$, and $\mathcal{M}$, $\widetilde{\mathcal{M}}$ as defined above. Then there exists a 
  generator matrix $G$ of an $[n'-r,k'-1,W]_q$ code $C$ such that the extension of $G$ via Lemma~\ref{lemma_ILP} yields at least one code that is isomorphic $C'$ and 
  the weight enumerator $w_C(x)$ is lexicographically minimal among the weight enumerators $w_{C_P}(x)$ for all $P\in\widetilde{\mathcal{M}}$ with column multiplicity $r$ in $C'$, 
  where $C_P$ is the linear code generated by the generator matrix $G_P$ from Lemma~\ref{lemma_shortening_refined}.
\end{corollary}
  
We remark that the construction for subcodes, as described in Lemma~\ref{lemma_shortening_refined}, can also be applied for points 
$P\in\mathcal{P}_k\backslash\mathcal{M}$. And indeed, we obtain an $[n-m(P),k-1]_q=[n,k-1]_q$ code, i.e., the effective length does not decrease, while the dimension decreases 
by one. 

The algorithmic implication of Corollary~\ref{cor_shortening_refined} is the following. Assume that we want to extend an $[n,k,W]_q$ code $C$ with generator matrix 
$G$ to an $[n+r,k+1,W]_q$ code $C'$ with generator matrix $G'$. If the minimum column multiplicity of $C$ is strictly smaller than $r$, then we do not need to 
compute any extension at all. Otherwise, we compute the set $\mathcal{S}(G)$ of solutions according to Lemma~\ref{lemma_ILP}. If a code $C'$ with generator matrix 
$G'$, corresponding to a solution in $\mathcal{S}(G)$, has a minimum column multiplicity which does not equal $r$, then we can skip this specific solution. For all other 
candidates let $\overline{\mathcal{M}}\subseteq \mathcal{P}_{k+1}$ the set of all different points spanned by the columns of $G'$ that have multiplicity exactly $r$. 
By our previous assumption $\overline{\mathcal{M}}$ is not the empty set. If $w_C(x)$ is the lexicographically minimal weight enumerator among all weight enumerators $w_{C_P}(x)$, 
where $P\in\overline{\mathcal{M}}$ and $C_P$ is generated by the generator matrix $G_P$ from Lemma~\ref{lemma_shortening_refined}, then we store $C'$ and skip it 
otherwise. We call the application of the extension step of Lemma~\ref{lemma_ILP} under these extra assumptions \emph{lexicographical extension} or \emph{lexicographical lengthening}.     

Lexicographical lengthening drastically decrease the ratio between the candidates of linear codes that have to be sifted out and the resulting number of non-isomorphic codes. 
This approach also allows parallelization of our enumeration algorithm, i.e., given an exhaustive list $\mathcal{C}$ of all $[n,k,W]_q$ codes and an integer $r\ge 1$, we can 
split $\mathcal{C}$ into subsets $\mathcal{C}_1,\dots,\mathcal{C}_l$ according to their weight enumerators. If the $[n+r,k+1,W]_q$ code $C'$ arises by lexicographical lengthening 
from a code in $\mathcal{C}_i$ and the $[n+r,k+1,W]_q$ code $C''$ arises by lexicographical lengthening from a code in $\mathcal{C}_j$, where $i\neq j$, then $C'$ and $C''$ 
cannot be isomorphic. As an example, when constructing the even $[21,8,6]_2$ codes from the $17\,927\,353$ $[20,7,6]_2$ codes, we can split the construction into more than 
$1000$ parallel jobs. If we do not need the resulting list of $1\,656\,768\,624$ linear codes for any further computations, there is no need to store the complete list of codes 
during the computation.

\section{Numerical results}
\label{sec_results}

As the implementation of a practically efficient algorithm for the classification of linear codes is a delicate issue, we exemplarily verify 
several classification results from the literature. Efficiency is demonstrated by partially extending some of these enumeration results. In
Subsection~\ref{subsec_applications} we show up some applications how exhaustive lists of linear codes can be used to find the extremal values 
of certain parameters of linear codes.

In \cite[Research Problem 7.2]{kaski2006classification} the authors ask for the classification of $[n,k,3]_2$ codes for $n>14$. In Table~\ref{tab_n_k_3_16}  
we extend their Table 7.7 to $n\le 16$.

\begin{table}[htp]
  \begin{center}
    \begin{tabular}{r|rrrrrrrrrrr}
      \hline
      $n/k$ & 1 & 2 & 3 & 4 & 5 & 6 & 7 & 8 & 9 & 10 & 11\\ 
      \hline
       3 & 1 \\ 
       4 & 1 \\ 
       5 & 1 & 1 \\ 
       6 & 1 & 3 & 1 \\									 
       7 & 1 & 4 & 4 & 1 \\								
       8 & 1 & 6 & 10 & 5 \\								
       9 & 1 & 8 & 23 & 23 & 5 \\							
      10 & 1 & 10 & 42 & 76 & 41 & 4 \\						
      11 & 1 & 12 & 71 & 207 & 227 & 60 & 3 \\					
      12 & 1 & 15 & 115 & 509 & 1012 & 636 & 86 & 2 \\				
      13 & 1 & 17 & 174 & 1127 & 3813 & 4932 & 1705 & 110 & 1 \\			
      14 & 1 & 20 & 255 & 2340 & 12836 & 31559 & 24998 & 4467 & 127 & 1 \\		
      15 & 1 & 23 & 364 & 4606 & 39750 & 176582 & 293871 & 132914 & 11507 & 143 & 1 \\	
      16 & 1 & 26 & 505 & 8685 & 115281 & 896316 & 2955644 & 3048590 & 733778 & 28947 & 144\\
      \hline 
    \end{tabular}
    \caption{The number of inequivalent $[n,k,3]_2$ codes for $n\le 16$}
    \label{tab_n_k_3_16}
  \end{center}
\end{table}

We remark that the entries \cite[Table 7.7]{kaski2006classification} are given for the number of $[\le n,k,3]_2$ codes in our notation, i.e., the numbers 
in Table~\ref{tab_n_k_3_16} above an entry have to be summed up to be directly compareable. Blank entries correspond to the non-existence of any code with 
these parameters, i.e., there is no $[4,2,3]_2$ code and also no $[16,12,3]_2$ code. Obviously, there is a unique $[n,1,3]_2$ codes for each $n\ge 3$ and 
it is not too hard to show that the number of inequivalent $[n,2,3]_2$ codes is given by
$
  \left\lceil\sqrt{\frac{(n-4)(n-3)(2n-7)}{6}}\,\right\rceil
$
for each $n\ge 3$. 
For each dimension $k\ge 1$ the maximum possible length $n$ of an $[n,k,3]_2$ code is also known. I.e., for each integer $r\ge 2$ there exists a unique 
$\left[2^r-1,2^r-r-1,3\right]_2$ code, which is called the $\left(2^r-1,2^r-r-1\right)$ \emph{Hamming code}. Other {\lq\lq}optimal{\rq\rq} codes can be obtained 
by shortening. E.g., there exist $[16+l,11+l,3]_2$ codes for $0\le l\le 15$. Their numbers are given by $144$, $129$, $113$, $91$, $67$, $50$, $34$, $21$, $14$, $9$, 
$5$, $3$, $2$, $1$, $1$, $1$. More precisely, not all these codes can be obtained by shortening, but we have completely classified them. In \cite{ostergaard2002classifying} 
also the number of inequivalent $[\le 15,7,3]_2$ codes was stated, which coincides with our enumeration. The entire computation of Table~\ref{tab_n_k_3_16} took less than 
11~hours of computation time on a single core of a 2.80GHz laptop bought in 2015. As said in \cite{ostergaard2002classifying}, it is not impossible to further extend the range of the 
classification, but we will focus on more interesting enumerations in order to demonstrate that also much larger numbers of codes can be classified. For completeness, we 
remark that we have also replicated the counts in tables 2,3 
from \cite{ostergaard2002classifying}.   

\begin{table}[htp]
  \begin{center}
    \begin{tabular}{c|ccccccccccc}
      \hline
      $k$ & 
      4 & 5 & 6 & 7 & 8 & 9 & 10 & 11 & 12 & 13 \\
      \hline
      \# & 
      8561\!&\!129586\!&\!1813958\!&\!16021319\!&\!60803805\!&\!73340021\!&\!22198835\!&\!1314705\!&\!11341\!&\!24 \\ 
      \hline
    \end{tabular}
    \caption{The number of inequivalent even $[\le 19,k,4]_2$ codes for $4\le k\le 13$}
    \label{tab_n_k_4_even}
  \end{center}
\end{table}

In \cite[Table 5]{bouyukliev2019classification} the counts of the even $[\le 18,k,4]_2$ codes are stated. We have verified these results and present the 
counts for the even $[\le 19,k,4]_2$ codes in Table~\ref{tab_n_k_4_even}. The counts of the even $[\le 20,k,6]_2$ codes are presented in 
\cite[Table 4]{bouyukliev2019classification}. We have verified these results and extended them to length $n\le 21$ in Table~\ref{tab_n_k_6_even} (excluding the 
enumeration of the even $[21,9,6]_2$ codes\footnote{Already the $17\,927\,353$ even $[20,7,6]_2$ codes can be extended to $1\,656\,768\,624$ even $[21,8,6]_2$ 
codes, so that we skipped the extension of the $39\,994\,046$ even $[20,8,6]_2$ codes.}). To turn these multitude of codes into something more manageable, we have 
used those results to classify all even $[k+10,k,6]_2$ codes. For $k\ge 12$ their numbers are given by $127$, $8$, and $1$, i.e., there is a unique even $[24,14,6]_2$ 
code, which is e.g.\ generated by
$$
\begin{pmatrix}
111111100010000000000000\\
000111111101000000000000\\
111011111100100000000000\\
001101100100010000000000\\
011010101000001000000000\\
110001110000000100000000\\
111101011000000010000000\\
101110001000000001000000\\
110110110100000000100000\\
101010110000000000010000\\
101011000100000000001000\\
100010011100000000000100\\
110101000100000000000010\\
101001101000000000000001
\end{pmatrix},
$$
has weight enumerator
$$
  w_C(x)=x^{0}+336x^{6}+1335x^{8}+3888x^{10}+5264x^{12}+3888x^{14}+1335x^{16}+336x^{18}+x^{24},
$$
and has an automorphism group of order $96$. The non-existence of a $[25,15,6]_2$ code is well-known \cite{simonis1987binary}. 

\begin{table}[htp]
  \begin{center}
    \begin{tabular}{c|ccccccccccc}
      \hline
      $k$ & 3 & 4 & 5 & 6 & 7 & 8 & 10 & 11  \\
      \hline
      \# & 726 & 12817 & 358997 & 11697757 & 246537467 & 1697180017 & 62180809 & 738 \\  
      \hline
    \end{tabular}
    \caption{The number of even $[\le 21,k,6]_2$ codes for $3\le k\le 11$, $k\neq 9$}
    \label{tab_n_k_6_even}
  \end{center}
\end{table}

For length $n=20$ the most time expensive step, i.e., extending the $[19,7,6]_2$ codes to $[20,8,6]_2$ codes, took roughly 
250~hours of computation time on a single core of a 2.80GHz laptop. We remark that the $[19,k,4]_2$ codes, where $k\in\{7,8,9,10\}$, and 
the $[21,k,6]_2$ codes, where $k\in\{7,8,10\}$, were enumerated in parallel, i.e., we have partially used the computing nodes of the 
\emph{High Performance Computing Keylab} from the University of Bayreuth. We have used the oldest cluster btrzx5 that went into operation in 
2009.\footnote{The precise technical details can be found at \url{https://www.bzhpc.uni-bayreuth.de/de/keylab/Cluster/btrzx5_page/index.html}.} 
This setup is chosen as an endurance test for our algorithm with hundred parallel jobs. During execution a few hard disks and CPUs died. We have 
tried our very best to detect possible hardware failures and to rerun all suspicious jobs. However, we are not 100\% sure that in those mentioned 
cases, which run on the computing cluster, the stated numbers are correct, which makes it a perfect opportunity for independent verification by 
other algorithms. 

\begin{table}[htp]
  \begin{center}
    \begin{tabular}{r|rrrrrrrrrrrrrr}
      \hline
      $n/k$ & 2 & 3 & 4 & 5 & 6 & 7 & 8 \\ 
      \hline  
      35 & 0 & 1 & 4 & 4 & 3 & 1 & 0\\	
      36 & 4 & 10 & 22 & 13 & 4 & 0& 0\\	
      37 & 0 & 2 & 7 & 10 & 3 & 1 & 0\\	
      38 & 0 & 1 & 6	& 12 & 10 & 3 & 1\\
      39 & 3 & 15 & 34 & 41 & 23 & 8 & 2\\
      40 & 0 & 6 & 25 & 40 & 30 & 10 & 1\\
      41 & 0 & 0 & 0	& 0 & 0 & 0 & 0\\
      42 & 2 & 17 & 52 & 44 & 15	& 0 & 0\\
      43 & 0 & 6 & 32 & 40 & 16 & 3 & 0\\
      44 & 0 & 2 & 14 & 22 & 17 & 6 & 1\\
      45 & 5 & 31 & 141 & 190 & 72 & 13 & 0\\
      46 & 0 & 6 & 56 & 122 & 71 & 18 & 3\\
      47 & 0 & 2 & 29 & 92 & 89 & 36 & 8\\
      48 & 5 & 44 & 297 & 705 & 468 & 128 & 28\\
      49 & 0 & 15 & 177 & 613 & 596 & 219 & 37\\
      50 & 0 & 2 & 39 & 217 & 295 & 149 & 40\\
      51 & 3 & 54 & 572 & 2405 & 2263 & 712 & 165\\
      52 & 0 & 18 & 333 & 1828 & 2909 & 1595 & 448\\
      53 & 0 & 6 & 116 & 1008 & 3512 & 3018 & 815\\
      54 & 8 & 91 & 1427 & 11121 & 23835 & 16641 & 2718\\
      55 & 0 & 19 & 651 & 4682 & 5839 & 1789 & 212\\
      \hline
    \end{tabular}
    \caption{The number of $9$-divisible $[n,k,9]_3$ codes for $35\le n\le 55$ and $2\le k\le 8$}
    \label{tab_n_k_9_div_q_3}
  \end{center}
\end{table}
          
Moreover, we have verified 
\begin{itemize}
  \item[-] the explicit numbers of the optimal binary codes of dimension~8 in \cite[Table 8]{bouyukliev2019classification};
  \item[-] the 
           enumerations results for the uniqueness of the $[46,9,20]_2$ code presented in \cite{kurz201946};
  \item[-] the enumeration of the projective $2$, $4$-, and $8$-divisible binary linear codes from \cite{ubt_eref40887};
  \item[-] the counts of $9$-divisible ternary codes in \cite[Table 6]{bouyukliev2019classification}; and
  \item[-] the counts of $4$-divisible quaternary codes in \cite[Table 7]{bouyukliev2019classification}.
\end{itemize}

Just to also have an extended example for a field size $q>2$ we have extended the results from \cite[Table 6]{bouyukliev2019classification} 
on $9$-divisible ternary codes to dimensions $k\le 8$ and length $n\le 55$, see Table~\ref{tab_n_k_9_div_q_3}. The conspicuous zero row 
for length $n=41$ has a theoretical explanation, i.e., there is no $9$-divisible $[41,k]_3$ code at all, see \cite[Theorem 1]{divisibleIEEE}.\footnote{More precisely,
$41=2\cdot 13+2\cdot 12-1\cdot 9$ is a certificate for the fact that such a code does not exist, see \cite[Theorem 1, Example 6]{divisibleIEEE}.}

\subsection{Applications}
\label{subsec_applications}

In this subsection we want to exemplarily show up, that exhaustive enumeration results of 
linear codes can of course be used to obtain results for special subclasses of codes and 
their properties by simply checking all codes. For our first example we remark that the 
support of a codeword is the set of its non-zero coordinates. A non-zero codeword $c$ of a 
linear code $C$ is called minimal if the support of no other non-zero codeword is contained 
in the support of $c$, see e.g.~\cite{ashikhmin1998minimal}. By $m_2(n,k)$ we denote the minimum 
number of minimal codewords of a projective\footnote{Duplicating columns in a binary 
linear code generated by the $k\times k$ unit matrix results in exactly $k$ minimal codewords, which 
is the minimum for all $k$-dimensional codes.} $[n,k]_2$ code. In Table~\ref{tab_minimal_codewords} 
we state the exact values of $m_2(n,k)$ for all $2\le k\le n\le 15$ obtained by enumerating all 
projective codes with these parameters.   
  
\begin{table}[htbp]
\begin{center}
{\small
\begin{tabular}{|c|c|c|c|c|c|c|c|c|c|c|c|c|c|c|}\hline
$n/k$ & 2 & 3   & 4     & 5     & 6     & 7     & 8     & 9  & 10 & 11 & 12 & 13 & 14 & 15 \\\hline
3     & 3 & 3   &       &       &       &       &       &    &    &    &    &    &    & \\\hline
4     &   & 4   & 4     &       &       &       &       &    &    &    &    &    &    & \\\hline
5     &   & 6   & 5     & 5     &       &       &       &    &    &    &    &    &    & \\\hline
6     &   & 7   & 6     & 6     & 6     &       &       &    &    &    &    &    &    & \\\hline
7     &   & 7   & 8     & 7     & 7     & 7     &       &    &    &    &    &    &    & \\\hline
8     &   &     & 8     & 9     & 8     & 8     & 8     &    &    &    &    &    &    & \\\hline
9     &   &     & 12    & 9     & 9     & 9     & 9     & 9 &    &    &    &    &    & \\\hline
10    &   &     & 14    & 10    & 10    & 10    & 10    & 10 & 10 &    &    &    &    & \\\hline
11    &   &     & 14    & 15    & 11    & 11    & 11    & 11 & 11 & 11 &    &    &    & \\\hline
12    &   &     & 15    & 15    & 13    & 12    & 12    & 12 & 12 & 12 & 12 &    &    & \\\hline
13    &   &     & 15    & 16    & 14    & 13    & 13    & 13 & 13 & 13 & 13 & 13 &    & \\\hline
14    &   &     & 15    & 16    & 14    & 15    & 14    & 14 & 14 & 14 & 14 & 14 & 14 & \\\hline
15    &   &     & 15    & 16    & 17    & 15    & 16    & 15 & 15 & 15 & 15 & 15 & 15 & 15\\
\hline\end{tabular}}
\end{center}
\caption{$m_2(n,k)$ for $3\leq n\leq 15, 1\leq k\leq 9$}
\label{tab_minimal_codewords}
\end{table}

In our second example we want to use the enumeration results from Table~\ref{tab_n_k_9_div_q_3} on 
ternary $9$-divisible linear codes. In \cite{no59} it was mentioned that the smallest length $n$ 
of a projective ternary $9$-divisible linear code whose existence is unknown is $n=70$. The possible 
weights are $9$, $18$, $27$, $36$, $45$, and $54$, since a codeword with weight $63$ would yield 
a projective $3$-divisible $[7,k]_3$ code, which does not exist, see \cite{divisibleIEEE}. Of course it 
is in principle possible to enumerate all $9$-divisible $[70,k]_3$ codes. However, there are already 
$85037$ such $[70,4]_3$ codes and their numbers explode with increasing dimension $k$. So, let us first 
derive some conditions on a hypothetical $9$-divisible $[70,k]_3$ code $C$. By $A_i$ we denote the number 
of codewords of weight $i$ of $C$ and by $B_i$ the number of codewords of weight $i$ of the dual code of 
$C$. The first equations of the well-known MacWilliams identities, see e.g.~\cite{macwilliams1977theory}, 
are given by:
\begin{eqnarray}
  1\!+\!A_9\!+\!A_{18}\!+\!A_{27}\!+\!A_{36}\!+\!A_{45}\!+\!A_{54} &\!\!\!\!\!=\!\!\!\!\!& 3^k \label{mw1}\\
  70\!+\!61 A_{9}\!+\!52 A_{18}\!+\!43 A_{27}\!+\!34 A_{36}\!+\!25 A_{45}\!+\!16 A_{54} &\!\!\!\!\!=\!\!\!\!\!& 70\cdot 3^{k\!-\!1} \label{mw2}\\
  2415\!+\!1830 A_9\!+\!1326 A_{18}\!+\!903 A_{27}\!+\!561 A_{36}\!+\!300 A_{45}\!+\!120 A_{54} &\!\!\!\!\!=\!\!\!\!\!& 2415\cdot 3^{k\!-\!2}\label{mw3}\\
  \!\!\!\!\!\!\!\!\!\!\!\!\!\!\!\!\!54740\!+\!35990 A_{9}\!+\!22100 A_{18}\!+\!12341 A_{27}\!+\!5984 A_{36}\!+\!2300 A_{45}\!+\!560 A_{54} &\!\!\!\!\!=\!\!\!\!\!& \left(54740\!+\!B_3\right)3^{k\!-\!3} \label{mw4}
\end{eqnarray}
$20$ times Equation~(\ref{mw1}) minus $2$ times Equation~(\ref{mw2}) plus $\tfrac{1}{10}$ times Equation~(\ref{mw3}) gives
$$
  \frac{3^5}{2} + 81 A_9 + \frac{243 A_{18}}{5} + \frac{243 A_{27}}{10} + \frac{81 A_{36}}{10} =\frac{3^k}{6},
$$ 
so that $k\ge 6$, since $A_i\ge 0$. For $k=6$ the polyhedron given by equations (\ref{mw1})-\ref{mw4}) and the nonegativity constraints 
$A_i,B_3\ge 0$ contains the unique point
$$
  A_9=A_{18}=A_{27}=A_{36}=0, A_{45}=588, A_{54}=140,\text{ and }B_3=280.
$$
However, a linear code $C$ with these parameters would be a $2$-weight code and the corresponding strongly regular graph 
does not exist, see e.g.~\cite{brouwerdistance} for the details. (We have also excluded this case by exhaustively 
enumerating the (non-existent) $[70,6,\{45,54\}]_3$ codes.) Thus, we can assume $k\ge 7$. For $k=7$ we can again 
consider the polyhedron given by equations (\ref{mw1})-\ref{mw4}) and the nonegativity constraints $A_i,B_3\ge 0$. 
Additionally we can assume that the $A_i$ are even integers. By solving the corresponding integer linear programs 
we can verify $A_9\le 2$, $A_{18}\le 4$, $A_{27}\le 10$, and $A_{36}\le 20$. Moreover, the first two constraints can 
be tightened to $2A_9+A_{18}\le 4$. We also can derive a condition on the length and the minimum column multiplicity, 
i.e., if a $9$-divisible $[n,k]_3$ code $C$ has minimum column multiplicity $\Lambda$ and $n+(7-k)\cdot \Lambda<70$, 
then $C$ cannot be extended to a $9$-divisible $[70,7]_3$ code via canonical lengthening, since in each extension 
step the length can increase by at most $\Lambda$. With those conditions we have performed a restricted generation of 
linear codes. We have indeed constructed a few hundred of $[69,6,\{9,18,27,36,45,54\}]_3$ codes with maximum column 
multiplicity $3$. However, none of these was extendable to a projective $9$-divisible $[70,7]_3$ code and we 
conjecture that no such code exists. Nevertheless, the above extra conditions drastically reduce the search space, 
it is still too large for our current implementation. In our computational experiments we have stopped the extension 
using \texttt{Solvediophant} after 10~minutes for each code, while we have seen unfinished lattice point enumerations 
lasting several hours. Moreover, we were not able to extend all $5$-dimensional codes due to their large number.

\section{Conclusion}
\label{sec_conclusion}
  
We have presented an algorithm for the classification of linear codes over finite fields based on lattice 
point enumeration. The lattice point enumeration itself and sifting out isomorphic copies is so far done 
with available scientific software packages. Using invariants like the weight enumerator of subcodes, see 
Corollary~\ref{cor_shortening_refined}, the number of candidates before sifting could kept reasonably small.  
The resulting algorithm is quite competitive compared to e.g.\ the recent algorithm described in 
\cite{bouyukliev2019classification}. There the authors used the appealing technique of canonical augmentation 
or orderly generation, see e.g.~\cite{royle1998orderly}. The advantage that no pairs of codes have to be checked 
whether they are isomorphic comes at the cost that the computation of the canonical form is relatively costly, 
see \cite{bouyukliev2019classification}. Allowing not only a single canonical extension, but a relatively small 
number of extensions that may lead to isomorphic codes, might be a practically efficient alternative. We have 
also demonstrated that the algorithm can be run in parallel.

However, we think that our implementation can still be further improved. In some cases the used lattice point 
enumeration algorithm \texttt{Solvediophant} takes quite long to verify that a certain code does not allow an 
extension, while integer linear programming solvers like e.g.~\texttt{Cplex} quickly verify infeasibility. 
Especially the computational experiments at the end of Subsection~\ref{subsec_applications} suggest, that it is 
worthwhile to try to speed up the lattice point enumeration. We propose the extension of Table~\ref{tab_n_k_9_div_q_3} 
as a specific open problem. 

Also it would be beneficial if at least some restriction of a lexicographical extension could be directly 
formulated as valid constraints in the integer linear programming formulation of Lemma~\ref{lemma_ILP}. So far 
we have not used known automorphisms of the linear code that should be extended. It is not implausible to expect 
that there for different parameter ranges different algorithmic choices can perform better. In any case, we have 
demonstrated that it is indeed possible to exhaustively classify sets of linear codes of magnitude $10^9$, which was 
not foreseeable  at the time of \cite{kaski2006classification}. 

Currently the implementation of the evolving software package \texttt{LinCode} is not that	progressed to be 
made publicy available. So, we would like to ask the readers to sent their interesting enumeration problems 
of linear codes to the author directly. 


\end{document}